\documentclass[a4paper,12pt]{article}
\usepackage{comment}
\usepackage{cite}
\usepackage{amsmath,amssymb,amsfonts}
\usepackage[T1]{fontenc}
\usepackage[utf8]{inputenc}
\usepackage{graphicx}
\usepackage{fancyhdr}
\usepackage{float}
\usepackage{xcolor}
\usepackage{authblk}
\usepackage{mathrsfs}
\usepackage[hyphens]{url}
\usepackage{hyperref}
\usepackage{amsthm}
\usepackage{tikz}
\usetikzlibrary{decorations.markings}
\usepackage{enumitem}
\usepackage{geometry}

\theoremstyle{plain}
\newtheorem{theorem}{Theorem}[section]
\newtheorem{proposition}[theorem]{Proposition}

\theoremstyle{definition}

\theoremstyle{remark}
\newtheorem{remark}[theorem]{Remark}

\newtheorem{example}{Example}[section]

\begin{document}

\title{Cumulative Riemann sums, distribution functions, and a universal inequality}

\author[$\dagger$]{Jean-Christophe {\sc Pain}$^{1,2,}$\footnote{jean-christophe.pain@cea.fr}\\
\small
$^1$CEA, DAM, DIF, F-91297 Arpajon, France\\
$^2$Universit\'e Paris-Saclay, CEA, Laboratoire Mati\`ere en Conditions Extrêmes,\\ 
F-91680 Bruy\`eres-le-Châtel, France
}

\date{}

\maketitle

\begin{abstract}
We study discrete expressions of the form
\[
T_n(g)=\sum_{i=1}^n a_i g(S_i), 
\qquad S_i=\sum_{j=1}^i a_j,
\]
where $a_i>0$ and $\sum_{i=1}^n a_i=1$. If $g:[0,1]\to\mathbb{R}$ is a decreasing integrable function, we have
\[
\sum_{i=1}^n a_i g(S_i) \le \int_0^1 g(x)\,dx,
\]
from which classical inequalities can be obtained, for instance for the choice $g(x)=1-x^k$. Although elementary, this inequality admits a natural interpretation in terms of Riemann sums, Abel summation, and the probability integral transform. The aim of this paper is to present a unified perspective emphasizing that the discrete inequality is a consequence of a distribution-free continuous identity. Beyond the specific example, we establish a general discrete theorem for monotone functions and discuss connections with majorization theory and Karamata's inequality.

\end{abstract}

\section{Introduction}

Let $a_1,\dots,a_n$ be positive numbers satisfying
\[
\sum_{i=1}^n a_i =1,
\]
and define the cumulative sums
\[
S_i = \sum_{j=1}^i a_j.
\]
Then we have
\[
0=S_0 < S_1 < \dots < S_n=1.
\]
Expressions involving the pair $(a_i,S_i)$ arise naturally in many contexts involving cumulative processes. In this article we study sums of the form
\[
T_n(g)=\sum_{i=1}^n a_i g(S_i).
\]

In Section \ref{sec:disc}, we first introduce the main discrete monotone inequality and provide a Riemann-sum interpretation that naturally leads to a continuous bound. In Section~\ref{sec:examples}, we illustrate the theorem with the classical polynomial case. We then present alternative formulations using Abel summation (in Section~\ref{sec:abel}) and the probability integral transform (in Section \ref{sec:conti}), highlighting the connection with continuous integrals and distribution functions. The connection with majorization is outlined in Section~\ref{sec:maj}, followed by a discussion on the strictness of the inequality in Section~\ref{sec:strict}, clarifying the conditions under which equality can or cannot occur. Possible applications in probability and numerical quadrature are mentioned in the conclusion.  In the appendix, we provide additional examples of decreasing functions, including exponential, logarithmic, reciprocal, and trigonometric functions, showing the generality of the inequality.

\section{Main discrete theorem and Riemann-sum interpretation}\label{sec:disc}

\begin{theorem}[Discrete monotone inequality]

Let $a_1,\dots,a_n>0$ with $\sum_{i=1}^n a_i=1$, and define
\[
S_i = \sum_{j=1}^i a_j.
\]
Let $g:[0,1]\to\mathbb{R}$ be a decreasing integrable function. Then
\[
\sum_{i=1}^n a_i g(S_i) \le \int_0^1 g(x)\,dx.
\]

\end{theorem}

\begin{proof}

Using $a_i=S_i-S_{i-1}$ we obtain
\[
\sum_{i=1}^n a_i g(S_i)
=
\sum_{i=1}^n (S_i-S_{i-1})g(S_i).
\]
Since $g$ is decreasing, we have
\[
g(S_i) \le g(x), \qquad x\in[S_{i-1},S_i],
\]
and we get
\[
(S_i-S_{i-1})g(S_i)
\le
\int_{S_{i-1}}^{S_i} g(x)\,dx.
\]
Summing over $i$ yields
\[
\sum_{i=1}^n a_i g(S_i)
\le
\sum_{i=1}^n \int_{S_{i-1}}^{S_i} g(x)\,dx
=
\int_0^1 g(x)\,dx.
\]

\end{proof}

Since
\[
a_i = S_i-S_{i-1},
\]
we may rewrite
\[
T_n(g)=\sum_{i=1}^n (S_i-S_{i-1})g(S_i).
\]
Thus $T_n(g)$ is a right-endpoint Riemann sum associated with the partition
\[
0=S_0<S_1<\dots<S_n=1.
\]
If $g$ is decreasing, each rectangle lies below the graph of $g$ and
\[
T_n(g) \le \int_0^1 g(x)\,dx.
\]

\section{A simple polynomial case}\label{sec:examples}

A motivating example is obtained for
\[
g(x)=1-x^k.
\]
We first rewrite the sum using cumulative sums. Since
\[
a_i = S_i - S_{i-1},
\]
we can write
\[
\sum_{i=1}^n a_i (1 - S_i^k) = \sum_{i=1}^n (S_i - S_{i-1}) (1 - S_i^k).
\]
This is exactly a right-endpoint Riemann sum for the integral \(\int_0^1 (1 - x^k)\, dx\) over the partition \(0 = S_0 < S_1 < \dots < S_n = 1\). The second step consists in bounding each term using monotonicity. The function \(g(x) = 1 - x^k\) is decreasing on \([0,1]\). Therefore, for any \(x \in [S_{i-1}, S_i]\), we can write
\[
1 - S_i^k \le 1 - x^k,
\]
and thus
\[
(S_i - S_{i-1}) (1 - S_i^k) \le \int_{S_{i-1}}^{S_i} (1 - x^k)\, dx.
\]
Summing over \(i=1,\dots,n\), we obtain
\[
\sum_{i=1}^n a_i (1 - S_i^k) \le \sum_{i=1}^n \int_{S_{i-1}}^{S_i} (1 - x^k)\, dx = \int_0^1 (1 - x^k)\, dx.
\]
The integral is elementary:
\[
\int_0^1 (1 - x^k) \, dx = \int_0^1 1 \, dx - \int_0^1 x^k \, dx
= 1 - \frac{1}{k+1} = \frac{k}{k+1}.
\]
Combining the previous steps, we get
\[
\sum_{i=1}^n a_i(1-S_i^k) < \frac{k}{k+1}.
\]

\begin{remark} 
The inequality is strict if all \(a_i>0\) and \(n>1\), since the right-endpoint Riemann sum strictly underestimates the integral for a strictly decreasing function, unless the partition coincides with specific points of the polynomial.
\end{remark}

\begin{example}
Let $n=3$ and $a_1=0.2$, $a_2=0.3$, $a_3=0.5$. Then $S_1=0.2$, $S_2=0.5$, $S_3=1$. For $g(x)=1-x^2$, we have
\[
\sum_{i=1}^3 a_i (1-S_i^2) = 0.2 (1-0.04) + 0.3 (1-0.25) + 0.5 (1-1) = 0.192 + 0.225 + 0 = 0.417 < \frac{2}{3}.
\]
\end{example}

Other examples (for different types of decreasing functions) are given in Appendix A.

\section{Abel summation}\label{sec:abel}

Another useful representation follows from discrete integration by
parts. Writing
\[
T_n = \sum_{i=1}^n (S_i-S_{i-1})g(S_i),
\]
Abel summation gives
\[
T_n = g(1) + \sum_{i=1}^{n-1} S_i\bigl(g(S_i)-g(S_{i+1})\bigr).
\]
In the polynomial case $g(x)=1-x^k$ this becomes
\[
T_n = \sum_{i=1}^{n-1} S_i (S_{i+1}^k-S_i^k).
\]
This discrete integration by parts, or Abel summation, provides an alternative representation of the sum 
\[
T_n = \sum_{i=1}^n a_i g(S_i)
\]
in terms of the cumulative sums \(S_i\) and the successive differences of \(g(S_i)\). It highlights how each term in the sum depends on both the cumulative weight \(S_i\) and the change in \(g\) between successive points. This form makes the role of the monotonicity of \(g\) explicit: if \(g\) is decreasing, then each term \(S_i (g(S_i)-g(S_{i+1}))\) is non-negative, providing an intuitive explanation for why the discrete sum is bounded above by the corresponding integral. Moreover, this representation is useful for analyzing strictness conditions and for potential generalizations to convex or piecewise functions.

\section{Continuous formulation and a general identity}\label{sec:conti}

Let $f$ be a probability density on $[0,1]$ and define
\[
F(x)=\int_0^x f(t)\,dt.
\]
Consider
\[
I = \int_0^1 f(x)[1-F(x)^k]dx.
\]
Since $F'(x)=f(x)$ we substitute
\[
u=F(x).
\]
Then
\[
du=f(x)dx
\]
and
\[
I=\int_0^1 (1-u^k)du=\frac{k}{k+1}.
\]
Thus
\[
\int_0^1 f(x)[1-F(x)^k]dx=\frac{k}{k+1}.
\]
The continuous formulation highlights the connection between the discrete inequality and a distribution-free integral identity. By interpreting \(f(x)\,dx\) as a probability measure and using the substitution \(u = F(x)\), the sum \(\sum a_i g(S_i)\) is seen as a discrete approximation of the integral \(\int_0^1 g(u)\,du\). This perspective shows that the upper bound is independent of the particular weights \(a_i\) and arises solely from the monotonicity of \(g\). It also provides a probabilistic interpretation: the integral corresponds to the expected value of \(g(U)\), where \(U\) is uniformly distributed on \([0,1]\) \cite{feller}, and the discrete sum approximates this expectation through cumulative sums.

\begin{proposition}

Let $f$ be a probability density on $[0,1]$ with cumulative distribution
function $F$. For any integrable function $g$,
\[
\int_0^1 f(x) g(F(x))dx
=
\int_0^1 g(u)du.
\]

\end{proposition}

This identity follows from the substitution $u=F(x)$.

\section{Connection with majorization}\label{sec:maj}

The cumulative vector $(S_1,\dots,S_n)$ possesses a natural ordering
\[
0<S_1<S_2<\dots<S_n=1.
\]
Such structures are closely related to the theory of majorization, which plays an important role in the theory of inequalities. Let $x=(x_1,\dots,x_n)$ and $y=(y_1,\dots,y_n)$ be vectors in $\mathbb{R}^n$, and let $x_{(1)} \ge x_{(2)} \ge \dots \ge x_{(n)}$ and $y_{(1)} \ge y_{(2)} \ge \dots \ge y_{(n)}$ denote their components arranged in decreasing order. One says that $x$ is majorized by $y$ (denoted $x \prec y$) if
\[
\sum_{i=1}^k x_{(i)} \le \sum_{i=1}^k y_{(i)}, \qquad k=1,\dots,n-1,
\]
and
\[
\sum_{i=1}^n x_i = \sum_{i=1}^n y_i .
\]
Intuitively, this means that the components of $x$ are more evenly distributed than those of $y$. A fundamental result in this area is Karamata's inequality: if $x \prec y$ and $g$ is a convex function, then (see Ref. \cite{hardy}):
\[
\sum_{i=1}^n g(x_i) \le \sum_{i=1}^n g(y_i).
\]
Although the framework considered in the present work
is different, the underlying philosophy is similar: structural constraints
on cumulative quantities lead to bounds that are independent of the
particular coefficients $a_i$ involved.

\section{Discussion on the strictness of the inequality}\label{sec:strict}

The discrete monotone inequality 
\[
\sum_{i=1}^n a_i g(S_i) \le \int_0^1 g(x)\,dx
\]
is strict if all weights $a_i>0$ and $n>1$, provided that $g$ is strictly decreasing on $[0,1]$. This is because a right-endpoint Riemann sum underestimates the integral for a strictly decreasing function, unless the partition coincides with points where the function has constant slope or a linear segment matching the rectangles exactly. Equality occurs in special, highly symmetric situations. For example:

\begin{itemize}
    \item If all $a_i = 1/n$ (uniform weights) and $S_i = i/n$, and $g$ is linear, then the Riemann sum coincides with the integral exactly:
    \[
    \sum_{i=1}^n \frac{1}{n} g\Big(\frac{i}{n}\Big) = \int_0^1 g(x)\,dx, \quad g(x)=mx+b.
    \]
    \item More generally, equality can hold if the partition points $(S_i)$ coincide with the points at which the integral of $g$ over each subinterval exactly equals the rectangle area used in the sum. This typically requires $g$ to have piecewise linear segments matching the cumulative sums.
\end{itemize}

Thus, for strictly decreasing nonlinear functions like $g(x)=1-x^k$, $g(x)=e^{-\lambda x}$, $g(x)=\ln(2-x)$, $g(x)=1/(1+x)$ or $g(x)=\cos(\pi x/2)$, equality never occurs for $n>1$ with all $a_i>0$ (see Appendix). This reinforces that the inequality is truly strict in most practical situations.

\section{Conclusion}

We have shown that expressions of the form
\begin{equation}\label{res}
T_n=\sum_{i=1}^n a_i g(S_i)
\end{equation}
where $a_i>0$, $\sum_{i=1}^n a_i=1$ and $g:[0,1]\to\mathbb{R}$ is a decreasing integrable function, admit a natural interpretation as Riemann approximations of a simple continuous identity arising from the probability integral transform. The discrete monotone inequality and its continuous counterpart admit several interesting applications. Let $X$ be a random variable uniformly distributed on $[0,1]$, and let $g$ be a decreasing function. Then the main theorem implies
\[
\sum_{i=1}^n a_i g(S_i) \le \int_0^1 g(x)\, dx = \mathbb{E}[g(X)],
\]
where $\mathbb{E}[g(X)]$ denotes the expected value of $g(X)$. This provides a natural upper bound for expectations of decreasing functions of cumulative sums, which can be useful in probability and statistics when approximating expectations from discrete data. The discrete sum in Eq. (\ref{res}) can be interpreted as a right-endpoint Riemann sum approximation of the integral $\int_0^1 g(x)\, dx$, he monotonicity of $g$ ensuring that this sum provides a guaranteed upper bound for the integral. Hence, the inequality can be used in numerical analysis to obtain safe estimates for integrals of decreasing functions when only non-uniform partitions or weighted sums are available. The cumulative vector $(S_1,\dots,S_n)$ naturally orders the partial sums of $(a_i)$. This viewpoint clarifies the origin of constants such as $k/(k+1)$ and suggests further connections with probabilistic structures. Although our setting deals with decreasing functions, monotone properties of $g$ combined with cumulative sums produce bounds independent of the particular coefficients $a_i$. This connects the inequality to classical results and provides a broader context for understanding discrete vs. continuous inequalities. It may also be helpful for the study of approximately convex functions \cite{Hyers1952,Green1952}.

\appendix

\section{Appendix: Inequalities for other functions}

In this appendix, we discuss four additional cases: exponential, logarithmic, reciprocal and trigonometric.

\begin{example}[Exponential case]
Consider the decreasing exponential function
\[
g(x) = e^{-\lambda x}, \quad \lambda > 0.
\]
By the main theorem, we have
\[
\sum_{i=1}^n a_i e^{-\lambda S_i} \le \int_0^1 e^{-\lambda x} \, dx = \frac{1-e^{-\lambda}}{\lambda}.
\]
This illustrates that the discrete inequality holds for non-polynomial functions as well.
\end{example}

\begin{example}[Logarithmic decreasing case]
Consider the decreasing logarithmic function
\[
g(x) = \ln(2-x), \quad x\in[0,1].
\]
Then the integral is
\[
\int_0^1 \ln(2-x)\, dx = \left[ (2-x)\ln(2-x) - (2-x) \right]_0^1 = 2\ln 2 - 1.
\]
Hence, the main theorem gives
\[
\sum_{i=1}^n a_i \ln(2-S_i) \le 2\ln 2 - 1.
\]
\end{example}

\begin{example}[Reciprocal decreasing case]
Consider the decreasing reciprocal function
\[
g(x) = \frac{1}{1+x}, \quad x \in [0,1].
\]
The integral is
\[
\int_0^1 \frac{1}{1+x} \, dx = \ln 2.
\]
Hence, the main theorem gives the discrete inequality
\[
\sum_{i=1}^n \frac{a_i}{1+S_i} \le \ln 2.
\]
This shows that the theorem applies to functions with a simple pole at $x=-1$, still decreasing on $[0,1]$.
\end{example}

\begin{example}[Trigonometric decreasing case]
Consider the decreasing trigonometric function
\[
g(x) = \cos\left(\frac{\pi x}{2}\right), \quad x \in [0,1].
\]
The integral is
\[
\int_0^1 \cos\left(\frac{\pi x}{2}\right) \, dx = \frac{2}{\pi}.
\]
Hence, the main theorem gives
\[
\sum_{i=1}^n a_i \cos\left(\frac{\pi S_i}{2}\right) \le \frac{2}{\pi}.
\]
This illustrates the applicability of the inequality to non-algebraic functions.
\end{example}

These examples show that the theorem applies to any decreasing monotone function, not only polynomials or exponentials.

\end{document}